\documentclass[11pt]{amsart}
\usepackage{amsmath}
\usepackage{amssymb}
\usepackage{amscd}

\def\NZQ{\mathbf}               
\def\NN{{\NZQ N}}
\def\QQ{{\NZQ Q}}
\def\ZZ{{\NZQ Z}}
\def\RR{{\NZQ R}}

\def\P{\mathcal P}
\def\R{{\mathbf R}}

\newtheorem{theorem}{Theorem}[section]

\newtheorem{lemma}[theorem]{Lemma}
\newtheorem{corollary}[theorem]{Corollary}
\theoremstyle{definition}
\newtheorem{definition}[theorem]{Definition}

\newtheorem{example}[theorem]{Example}
\theoremstyle{remark}

\theoremstyle{Remark}
\newtheorem{Remark}[theorem]{Remarks}

\newtheorem{Theorem}{Theorem}[section]

\newtheorem{Example}[Theorem]{Example}

\let\epsilon\varepsilon
\let\phi=\varphi
\let\kappa=\varkappa
\begin{document}

\title{Semigroups of valuations on local rings, II}

\author{Steven Dale Cutkosky, Bernard Teissier}
\thanks{The first author was partially supported by NSF and by the University Paris 7-Denis Diderot.\\ AMS classification: Primary: 13A18, 14 E15, 16W50.
Secondary: 06F05}

\maketitle
\begin{abstract}
Given a noetherian local domain $R$ and a valuation $\nu$ of its field of fractions which is non negative on $R$, we derive some very general bounds on the growth of the number of distinct valuation ideals of $R$ corresponding to values lying in certain parts of the value group $\Gamma$ of $\nu$. We show that this growth condition imposes restrictions on the semigroups $\nu(R\setminus \{0\})$ for noetherian $R$ which are stronger that those resulting from the previous paper \cite{C2} of the first author. Given an ordered embedding $\Gamma\subset ({\mathbf R}^h)_{\hbox{\rm lex}}$, where $h$ is the rank of $\nu$, we also study the shape in ${\mathbf R}^h$ of the parts of $\Gamma$ which appear naturally in this study. We give examples which show that this shape can be quite wild in a way which does not depend on the embedding and suggest that it is a good indicator of the complexity of the semigroup $\nu(R\setminus \{0\})$.
\end{abstract}
\par\medskip
Let $(R,m_R)$ be a local domain, with fraction field $K$. Suppose that $\nu$ is a valuation of $K$
with valuation ring $(V,m_V)$ which dominates $R$; that is, $R\subset V$ and $m_V\cap
R=m_R$. The value groups $\Gamma$ of $\nu$ which can appear when $K$ is an algebraic function field have been extensively studied and classified,
including in the papers  MacLane \cite{M}, MacLane and Schilling \cite{MS}, Zariski and Samuel
\cite{ZS}, Kuhlmann \cite{K} and Moghaddam \cite{Mo}. These groups are well understood. The most basic fact is that there is an order preserving
embedding of $\Gamma$ into $\RR^h$ with the lex order, where $h$ is the \textit{rank} of the valuation, which is less than or equal to the dimension of $R$. The
semigroups
$$
S^R(\nu)=\{\nu(f)\mid f\in m_R-\{0\}\},
$$
which can appear when $R$ is a noetherian domain with fraction field $K$ dominated by $\nu$,  are not well understood, although they are known to encode important information about the  ideal theory of $R$ and the
geometry and resolution of singularities of $\hbox{\rm Spec }R$. In particular, after \cite{T}, the toric resolutions of singularities of the affine toric varieties associated to certain finitely generated subsemigroups of $S^R(\nu)$ are closely related to the local uniformizations of $\nu$ on $R$.

In Zariski and Samuel's classic book on Commutative Algebra \cite{ZS}, two general facts about
semigroups $S^R(\nu)$ of valuations on noetherian local domains are proven (in Appendix 3 to Volume II).
\begin{enumerate}
\item[1.] For any valuation $\nu$ of $K$ which is non negative on $R$, the semigroup $S^R(\nu)$ is a well ordered subset of the positive part of the value group $\Gamma$ of
$\nu$, of ordinal type at most $\omega^h$, where $\omega$ is the ordinal type of the well ordered
set $\NN$, and $h$ is the rank of the valuation. \item[2.] If $\nu$ dominates $R$, the rational rank of $\nu$ plus the
transcendence degree of $V/m_V$ over $R/m_R$ is less than or equal to the dimension of $R$.
\end{enumerate}
The second condition is the Abhyankar inequality \cite{Ab}.

In \cite{CT}, the authors give some examples showing that some surprising semigroups of
rank $>1$ can occur as   semigroups of valuations on noetherian domains, and raise the general
question of finding new constraints on value semigroups and classifying semigroups which occur as
value semigroups.

The only semigroups which are realized by a valuation on a one dimensional regular local ring are
isomorphic to the natural numbers. The semigroups which are realized by a valuation on a regular
local ring of dimension 2 with algebraically closed residue field are much more complicated, but
are completely classified by Spivakovsky in \cite{S}. A different proof is given by Favre and
Jonsson in \cite{FJ}, and we reformulated the theorem in the context of semigroups in \cite{CT}. However, very little is known in higher dimensions. The classification of
semigroups of valuations on regular local rings of dimension two does suggest that there may be
constraints on the rate of growth of the number of new generators on semigroups of valuations
dominating a noetherian domain. In \cite{C2}, such a constraint is found for rank 1 valuations. We prove in this paper that there is such a constraint for valuations of arbitrary rank.

In \cite{C2}, a very simple polynomial bound is found on the growth of $S^R(\nu)$ for a rank 1 valuation $\nu$. This bound allowed the 
construction  in \cite{C2} of a well ordered subsemigroup of $\QQ_+$ of ordinal type
$\omega$, which is not a value semigroup of a noetherian local domain. Thus the above conditions 1
and 2 do not characterize value semigroups on  local domains.
\par\vskip .2truein\noindent
$\bullet$\textit{Unless otherwise stated, in this text all local rings are assumed to be noetherian. A valuation of a local domain is a valuation $\nu$ of its field of fractions whose ring $R_\nu$ contains $R$ in such a way that $m_\nu\cap R\subseteq m_R$.} 
\vskip .2truein
In Section 1 of this paper, we describe a polynomial behavior of valuation ideals $\P_\phi(R)=\{x\in R\vert \nu(x)\geq \phi\}$ and $\P_\phi^+(R)=\{x\in R\vert \nu(x)> \phi\}$.
Given a valuation $\nu$ with center $p$ on a local domain $R$ we find very general polynomial bounds on the growth of the sums of the multiplicities of the finitely generated $R/p$-modules $\P_\phi(R)/\P^+_\phi(R)$ when $\phi$ runs through growing regions of the value group $\Gamma$ of $\nu$ viewed as a subgroup of $\R^h$. These modules are nonzero precisely when $\phi$ is $0$ or belongs to the semigroup $S^R(\nu)$. Our results therefore also bound the number of elements of $S^R(\nu)$ in those regions. This last result generalizes to all ranks the bound given for rank 1 valuations in \cite{C2} (restated as Theorem \ref{Theorem01} in this paper). The statement and proof for higher rank valuations is significantly more complex.\par
We give  an example (Example \ref{Example3}) of a rank 2 semigroup $T$ which satisfies all restrictions on the semigroup of a valuation on an $s$ dimensional local domain imposed by our polynomial bounds for modules over the rank 1 convex subgroup $\Phi_1$ of the group $\Gamma$ generated by $T$, but is not a valuation semigroup on an $s$ dimensional local domain. The proof uses our most general bound, Theorem \ref{Theorem1}, in the case of  rank 2 valuations.

Our polynomial bounds are estimates of sums over the intersection of $S^R(\nu)$ with certain regions of $\Gamma$. These regions are defined by their intersections with the convex subgroups of $\Gamma$ and depend on a certain function $\tilde\phi$ whose precise definition is given in Definition \ref{Def1} of Section 1. Given a valuation $\nu_{i+1}$ composed with $\nu$ and an element $\phi$ in $\nu_{i+1}(R\setminus \{0\})\subset \Gamma/\Phi_i$, the value $\tilde\phi$ is the smallest element in the semigroup $S^R(\nu_i)$ which projects to $\phi$; it is an element of $\Gamma/\Phi_{i-1}$. A lower limit in the sum at level $i$ is determined by the values traced out by $\tilde\phi$ as $\phi$ varies in the semigroup $\nu_{i+1}(R\setminus \{0\})$, while the upper limit is of the form $\tilde\phi+y_it_i,\ y_i\in \NN$. It is interesting
to consider how close these regions are to being polydiscs. The most desirable
situation is when the value group can be embedded by an order preserving homomorphism into $(\RR^h)_{\mbox{lex}}$ so that all of these regions are polydiscs. In the first examples that one is likely to consider, this is in fact the case. However, the general situation is not so simple. In 
Section 2 we give examples showing that the tilde function can exhibit a rather wild behavior. We show that  we can make $\tilde\phi$ decrease arbitrarily fast, and  that this is independent of the embedding. We also show that $\tilde \phi$ can increase arbitrarily fast, and finally that $\tilde\phi$ can jump back and forth from  negative numbers which decrease arbitrarily fast to 
positive  numbers which increase arbitrarily fast. All these examples are independent of the embedding of $\Gamma$ into $\R^h$. In view of the results of Section 1, the behavior of $\tilde\phi$ is an interesting measure of the complexity of the valuation.

\section{Polynomial bounds on  valuation ideals}
In this section, we derive some very general bounds on the growth of the number of distinct valuation ideals corresponding to values lying in certain parts of the group $\Gamma$.

If $G$ is a totally ordered abelian group, then $G_+$ will denote the positive elements of $G$, and $G_{\ge 0}$ will denote the nonnegative elements.
If $R$ is a local ring, $m_R$ will denote its maximal ideal, and   $\mbox{length}_R(N)$ will denote the length
of an $R$-module $N$.\par
Suppose that $R$ is a domain and $\nu$ is a valuation of $R$. Let $\Gamma$ be the value group of $\nu$. We will denote 
the value semigroup of $\nu$ on $R$ by
$$
S^R(\nu)=\{\nu(f)\mid f\in m_R-\{0\}\}.
$$
$S^R(\nu)$ is a subsemigroup of the nonnegative part, $\Gamma_{\ge 0}$, of $\Gamma$, and if $\nu$ dominates $R$, so that all elements of the maximal ideal $m_R$  of $R$ have positive value, then $S^R(\nu)$ is a subsemigroup of the semigroup $\Gamma_{+}$ of positive elements of $\Gamma$.

Suppose that $I\subset R$ is an ideal. We will write
$$
\nu(I)=\mbox{min}\{\nu(f)\mid f\in I-\{0\}\}.
$$
Note that $\nu(I)\in\Gamma_{\ge 0}$ exists since  $R$ is noetherian.

Suppose that  $\phi$ is an element of the value group $\Gamma$. We will denote by $\P_\phi(R)$ the ideal $\{x\in R\mid \nu(x)\geq \phi\}$ and by $\P^+_\phi(R)$ the ideal $\{x\in R\mid \nu(x)> \phi\}$. When no confusion on the ring is possible we will write $\P_\phi, \P^+_\phi$. We note that $\P_{\phi}(R)/\P_{\phi}^+(R)=0$ if and only if $\phi\notin S^R( \nu)\cup\{0\}$.
The associated graded ring of $\nu$ on $R$ is
$$\mbox{gr}_{\nu}(R)=\bigoplus_{\phi\in\Gamma}\P_{\phi}(R)/\P_{\phi}^+(R).$$
This $(R/m_\nu\cap R)$-algebra is not in general finitely generated; it is graded by the semigroup $S^R(\nu)$, which is not finitely generated in general. Our results can be seen as an extension to these algebras of the classical results on $\NN$-graded finitely generated algebras.

Suppose that $\Gamma$ is a totally ordered abelian group, and $a,b\in\Gamma$. We set
$$
[a,b]=\{x\in\Gamma\mid a\leq x\leq b\}\ \hbox{\rm and}\ [a,b[=\{x\in\Gamma\mid a\leq x< b\}
$$

The concepts of rank of a valuation, the convex (isolated) subgroups of a valuation group and the corresponding composed valuations are discussed in detail in 
Chapter VI of \cite{ZS}.

Suppose that $\nu$ has rank $n$. Let
$$
0=\Phi_0\subset \Phi_1\subset \cdots\subset \Phi_n=\Gamma
$$
be the sequence of convex subgroups of $\Gamma$. Let $\nu_i$, for $1\le i\le n$, be the valuations on the quotient field of $R$ with which $\nu$ is composed.
We have $\nu_1=\nu$. Let
$$
p_n\subseteq \cdots \subseteq p_1
$$
be the corresponding centers on $R$ of $\nu_i$. We define $p_{n+1}=(0)$.

The value group of $\nu_i$ is 
$\Gamma_i=\Gamma/\Phi_{i-1}$.
For $1\le i\le n$, set $$t_i=\nu_i(p_i)\in \Phi_i/\Phi_{i-1}\subseteq\Gamma/\Phi_{i-1}.$$ Let
$$
\lambda_i:\Gamma_i=\Gamma/\Phi_{i-1}\rightarrow \Gamma_{i+1}=\Gamma/\Phi_i
$$
be the corresponding maps from the value group of $\nu_i$ to the value group of $\nu_{i+1}$. If $p_i\ne p_{i+1}$, then $t_i$ is in the kernel of $\lambda_i$, which is a rank 1 group. 
When there is no ambiguity, we denote by $\phi_i$ the image in $\Gamma/\Phi_{i-1}$ of an element $\phi\in \Gamma$.

\begin{definition}\label{Def1}
 Given $\phi_i\in\Gamma/\Phi_{i-1}$, denote by $\tilde\phi_i\in\Gamma/\Phi_{i-2}$ the minimum of $\nu_{i-1}(f)$ for $f\in R$ such that $\nu_i(f)=\phi_i$. 
\end{definition}
This minimum exists since the semigroup $S^R(\nu_{i-1})$ is well ordered. 
Note that $\lambda_{i-1}(\tilde\phi_i)=\phi_i$.

\par\medskip

If $p_{i-1}\ne p_i$, we remark that for $y_{i-1}\in\NN$ and $\phi_{i-1}\in [\tilde \phi_i,\tilde\phi_i+y_{i-1}t_{i-1}]$, we have the inclusions:\par\noindent
$$p_{i-1}^{y_{i-1}}\P_{\phi_i}\subset \P_{\phi_{i-1}}\subset \P_{\phi_i},$$ 
$\P_{\phi_i}=\P_{\tilde \phi_i}$ and since $\Phi_i/\Phi_{i-1}$  is of rank one, the number of elements of $\nu_{i-1}(R\setminus\{0\})$ in the interval $ [\tilde \phi_i,\tilde\phi_i+y_{i-1}t_{i-1}]$ is finite (see \cite{ZS}, \textit{loc. cit.}).\par

\begin{lemma}\label{basic}  Suppose that $p_1\ne p_2$. Then for any function $A$ on $R$-modules with values in $\RR$ which is additive on short exact sequences of finitely generated $R$-modules whose unique minimal prime is $p_1$, we have  for all $y_1\in\NN$: 
\begin{equation}\label{eqA1}
\sum_{\phi_1\in[\tilde \phi_2,\tilde\phi_2+y_1t_1[}A(\P_{\phi_1}/\P_{\phi_1}^+)\leq A(M_{\phi_2}/p_1^{y_1}M_{\phi_2}),
\end{equation}
where $M_{\phi_2}= \P_{\phi_2}/ \P^+_{\phi_2},$ a finitely generated torsion free $R/p_2$-module.\end{lemma}

\begin{proof} 
For $y_1\in\NN$, $[\tilde\phi_2,\tilde\phi_2+t_1y_1[$ intersects $S^R(\nu)\cup\{0\}$ in a finite set $\{\tau_1,\ldots,\tau_r\}$, with 
$$
\tau_1=\tilde\phi_2<\tau_2<\cdots<\tau_r<\tilde\phi_2+t_1y_1.
$$
We have inclusions of $R$ modules whose unique minimal prime is $p_1$,
$$
\P_{\tau_r}/\P_{\tilde\phi_2+t_1y_1}\subset
\P_{\tau_{r-1}}/\P_{\tilde\phi_2+t_1y_1}
\cdots\subset\P_{\tau_1}/\P_{\tilde\phi_2+t_1y_1}=\P_{\phi_2}/\P_{\tilde\phi_2+t_1y_1}.
$$
By the additivity of $A$ we have $$A(\P_{\phi_2}/\P_{\tilde \phi_2+t_1y_1})=\sum_{\phi_1\in[\tilde \phi_2,\tilde\phi_2+y_1t_1[}A(\P_{\phi_1}/\P_{\phi_1}^+).$$ 
From the inclusion $p_1^{y_1}\P_{\phi_2}\subset\P_{\tilde\phi_2+y_1t_1}$, we have an exact sequence of $R$-modules whose unique minimal prime is $p_1$:
$$
0\rightarrow
\P_{\tilde\phi_2+t_1y_1}/(\P_{\phi_2}^++p_1^{y_1}\P_{\phi_2})
\rightarrow \P_{\phi_2}/(\P_{\phi_2}^++p_1^{y_1}\P_{\phi_2})
\rightarrow \P_{\phi_2}/\P_{\tilde\phi_2+t_1y_1}\rightarrow 0.
$$
Since 
$$
M_{\phi_2}/p_1^{y_1} M_{\phi_2}\cong\P_{\phi_2}/(\P_{\phi_2}^++p_1^{y_1}\P_{\phi_2}),
$$
we have that 
$$
A(\P_{\phi_2}/\P_{\tilde\phi_2+t_1y_1})\le A(M_{\phi_2}/p_1^{y_1} M_{\phi_2}),
$$
and the conclusions of the lemma follow.
\end{proof}

Suppose that $p_0$ is a prime ideal of $R$ such that $p_2\subset p_1\subset p_0$. Let $e_{m_i}(N)$ denote the multiplicity of an $R_{p_i}$ module $N$ with respect to $m_i=p_iR_{p_i}$.

From the above lemma, we immediately deduce the following result. 

\begin{Theorem}\label{Theorem01} Let $R$ be a local domain and $\nu$ a rank 1 valuation of $R$. Let $p_0$ be a prime ideal ideal of $R$ containing the center $p_1$ of $\nu$.
  Then
$$
\sum_{\phi\in[0,yt_1[}e_{m_0}((\P_{\phi}/\P_{\phi}^+)_{p_0})
\le e_{m_0}((R/p_1)_{p_0})\hbox{\rm length}_{R_{p_1}}(R_{p_1}/p_1^yR_{p_1})
$$
for all $y\in\NN$. Thus we have
$$
\sum_{\phi\in[0,yt_1[}e_{m_0}((\P_{\phi}/\P_{\phi}^+)_{p_0})
\le e_{m_0}((R/p_1)_{p_0})P_{R_{p_1}}(y)
$$
for $y\gg 0$, where $P_{R_{p_1}}(y)$ is the Hilbert-Samuel polynomial of the local ring $R_{p_1}$.
\end{Theorem}

\begin{proof}
This follows from Lemma \ref{basic}, with $A(N)=e_{m_0}(N\otimes R_{p_0})$ for $R$-modules $N$. We take $\nu=\nu_1$, $\Gamma_2=0$, $\phi_2=0$, and $p_1\subset p_0$, so that $\tilde\phi_2=0$ and $M_{\phi_2}=R$, to get for $y\in\NN$,
$$
\sum_{\phi\in[0,yt_1[}e_{m_0}((\P_{\phi}/\P_{\phi}^+)_{p_0})
\le e_{m_0}(R_{p_0}/p_1^yR_{p_0}).
$$
The conclusions of the theorem now follow from the associativity formula for
multiplicity, \cite{B}, Section 7, no. 1, Proposition 3, which shows that
$$
e_{m_0}(R_{p_0}/p_1^yR_{p_0})=e_{m_0}((R/p_1)_{p_0})
e_{m_1}(R_{p_1}/p_1^yR_{p_1})
= e_{m_0}((R/p_1)_{p_0})\hbox{\rm length}_{R_{p_1}}(R_{p_1}/p_1^yR_{p_1}).
$$

\end{proof}

If $\nu$ has rank 1, and dominates $R$, so that $p_1=p_0=m_R$ is the maximal ideal of $R$, we obtain the inequality of \cite{C2},
\begin{equation}\label{eqL12}
\#(S^R(\nu)\cap [0,yt_1[\,)<P_R(y)
\end{equation} 
for $y\in\NN$ sufficiently large. This follows from Theorem \ref{Theorem01}
since
 $\P_{\phi}/\P_{\phi}^+\ne 0$ if and only if
there exists $f\in R$ such that $\nu(f)=\phi$, and since $\phi=0$ is not
in $S^R(\nu)$.

As was shown in \cite{C2}, we may now easily construct a  well ordered subsemigroup $U$ of $\QQ_+$ such that $U$ has ordinal
type $\omega$ and $U\ne S^R(\nu)$ for any valuation $\nu$ dominating a local domain $R$. We let $T$ be any subset of $\QQ_+$ which has 1 as its smallest element, and
$$
y^y<\#([y,y+1[\,)\cap T)<\infty
$$
for all $y\in\ZZ_+$. Let $U=\bigcup_{n=1}^{\infty}nT$ be the semigroup generated by $T$. Then $U$ is well ordered by a result of B. H. Neumann (see \cite{Ne}), and the function $\#([0,y[\,\,\cap U)$ grows faster than $y^d$ for any $d\in\NN$. Since the Hilbert-Samuel polynomial of a noetherian local domain $R$ has degree $d=\hbox{\rm dim }R<\infty$, it follows from formula (\ref{eqL12}) that $U$ cannot be the semigroup of a valuation dominating a noetherian local domain.

Suppose now that $p_0$ is a prime ideal of $R$ such that $p_2\subseteq p_1\subseteq p_0$. Let $e_{m_i}(N)$ denote the multiplicity of an $R_{p_i}$ module $N$ with respect to $m_i=p_iR_{p_i}$.\par
If $p_2=p_1$, so that $t_1$ is not in the kernel of $\lambda_1$, let us define
$$
\phi_2^+=\mbox{min}\{\nu_2(f)\mid f\in R\mbox{ and }\nu_2(f)>\phi_2\}.
$$ 
By \cite{ZS} (Appendix 3, Corollary to Lemma 4), the interval $[\tilde \phi_2,\widetilde{\phi_2^+}[$ contains only finitely many elements of $S^R(\nu)$.
\begin{theorem}\label{TheoremA2}\par\noindent Let $R$ be a local domain and $\nu_1, \nu_2$ two valuations of $R$ such that $\nu_1$ is composed with $\nu_2$ and the difference of their ranks is equal to one. Let $p_0$ be a prime ideal of $R$ containing the centers $p_2\subseteq p_1$ of $\nu_1$ and $\nu_2$. Then we have:\par\noindent  a) Suppose that $p_1\ne p_2$. Then there exists 
a function $s(\epsilon,\phi_2)$ such that for $\phi_2\in\Gamma_2$, $\epsilon>0$ and $y_1\in\NN$ such that $y_1>s(\epsilon,\phi_2)$, we
have
$$
\begin{array}{l}
\sum_{\phi_1\in[\tilde \phi_2,\tilde\phi_2+y_1t_1[}e_{m_0}(\P_{\phi_1}(R_{p_0})/\P_{\phi_1}^+(R_{p_0}))\\ \leq
(1+\epsilon)\frac{e_{m_0}((R/p_1)_{p_0})}{(\hbox{\rm dim }(R/p_2)_{p_1})!}e_{m_1}(\P_{\phi_2}(R_{p_1})/\P_{\phi_2}^+(R_{p_1}))y_1^{\hbox{\rm dim }(R/p_2)_{p_1}}.
\end{array}
$$
b) If $p_1=p_2$ we have the equalities $$\begin{array}{ll}
\sum_{\phi_1\in[\tilde\phi_2,\widetilde{\phi_2^+[}}e_{m_0}(\P_{\phi_1}(R_{p_0})/\P_{\phi_1}^+(R_{p_0}))&= e_{m_0}(\P_{\phi_2}(R_{p_0})/\P_{\phi_2}^+(R_{p_0}))\\&=e_{m_0}((R/p_1)_{p_0})e_{m_1}(\P_{\phi_2}(R_{p_1})/\P_{\phi_2}^+(R_{p_1}))
.\end{array}$$ 
\end{theorem}

\begin{proof} Assume first that $p_1\neq p_2$. Taking $A(N)=e_{m_0}(N_{p_0})$ in Lemma \ref{basic},
and using the identities $\P_{\phi_1}(R_{p_0})\cong (\P_{\phi_1})_{p_0}$,
we obtain
$$
\sum_{\phi_1\in[\tilde \phi_2,\tilde\phi_2+y_1t_1[}e_{m_0}(\P_{\phi_1}(R_{p_0})/\P_{\phi_1}^+(R_{p_0}))\leq
e_{m_0}((M_{\phi_2})_{p_0}/p_1^{y_1} (M_{\phi_2})_{p_0}).
$$
Since $(p_1)_{p_0}$ is the unique  minimal prime of  $(M_{\phi_2})_{p_0}/p_1^{y_1}(M_{\phi_2})_{p_0}$, by \cite{B},  Section 7, no. 1, Proposition 3, we have
$$
e_{m_0}((M_{\phi_2})_{p_0}/p_1^{y_1} (M_{\phi_2})_{p_0})
=\hbox{\rm length}_{R_{p_1}}((M_{\phi_2})_{p_1}/p_1^{y_1} (M_{\phi_2})_{p_1})e_{m_0}((R/p_1)_{p_0}).
$$
Since  $p_2$ is the unique minimal prime of  $M_{\phi_2}$, there exists a function $\overline s(\phi_2)$ such that 
$$
\begin{array}{lll}
\hbox{\rm length}_{R_{p_1}}((M_{\phi_2})_{p_1}/p_1^{y_1} (M_{\phi_2})_{p_1})
&=&H_{(M_{\phi_2})_{p_1}}(y_1)\\
&=&\frac{e_{m_1}((M_{\phi_2})_{p_1})}{(\hbox{\rm dim}(R/p_2)_{p_1})!}
y_1^{\hbox{\rm dim}(R/p_2)_{p_1}}+\hbox{\rm lower order terms in $y_1$}
\end{array}
$$
for $y_1\ge\overline s(\phi_2)$, where $H_{(M_{\phi_2})_{p_1}}(y_1)$
is the Hilbert-Samuel polynomial of $(M_{\phi_2})_{p_1}$.
This polynomial bound implies that there exists a function $s(\epsilon,\phi_2)$ such that
$$
\mbox{length}_{R_{p_1}}((M_{\phi_2})_{p_1}/p_1^{y_1}(M_{\phi_2})_{p_1})\le (1+\epsilon)\frac{e_{m_1}((M_{\phi_2})_{p_1})}{(\hbox{\rm dim}(R/p_2)_{p_1})!}
y_1^{\hbox{\rm dim}(R/p_2)_{p_1}}
$$
for $y_1\ge s(\epsilon,\phi_2)$.\par\medskip 
If $p_1=p_2$, 
 We have $p_1(\P_{\phi_2}/\P_{\phi_2}^+)=p_2(\P_{\phi_2}/\P_{\phi_2}^+)=0$, so the first inequality stated in this case follows directly from the additivity of the multiplicity $e_{m_0}$. The second equality follows from the first and the associativity formula of \cite{B},  Section 7, no. 1, Proposition 3.
\end{proof}
\begin{corollary}
Suppose that $p_n=\cdots=p_2=p_1$.  Then
$$
\sum_{\phi_n\in[0,yt_n[}\sum_{\phi_{n-1}\in[\tilde\phi_n,\tilde\phi_n^+[}\cdots\sum_{\phi_1\in[\tilde\phi_2,\tilde\phi_2^+[} e_{m_0}((\P_{\phi_1}/\P_{\phi_1}^+)
\le e_{m_0}((R/p_1)_{p_0})\mbox{length}_{R_{p_1}}(R_{p_1}/p_1^yR_{p_2})
$$
for all $y\in\NN$. Thus we have
$$
\sum_{\phi_n\in[0,yt_n[}\sum_{\phi_{n-1}\in[\tilde\phi_n,\tilde\phi_n^+[}\cdots\sum_{\phi_1\in[\tilde\phi_2,\tilde\phi_2^+[} e_{m_0}((\P_{\phi_1}/\P_{\phi_1}^+)
\le e_{m_0}((R/p_1)_{p_0})P_{R_{p_1}}(y)
$$
for $y\gg 0$, where $P_{R_{p_1}}(y)$ is the Hilbert-Samuel polynomial of $R_{p_1}$.

\end{corollary}
\begin{proof}
We will prove the formula by induction on the rank $n$ of the valuation. If 
$n=1$, this is just the statement of Theorem \ref{Theorem01}.
We will assume that the formula is true for valuations of rank $<n$, and derived the formula for a rank $n$ valuation $\nu$. Let $\nu_2$ be the rank $n-1$ valuation which $\nu$ is composite with. Consider the chain of ideals
$$
(0)=q_n\subset q_{n-1}=\cdots =q_1=q_0,
$$
where $q_{n-1}=p_n,\ldots,q_1=p_2$ are the centers on $R$ of the successive valuations $\nu_n,\ldots,\nu_2$ with which $\nu_2$ is composed, and $q_0=p_1$.
We obtain
\begin{equation}\label{eqL10}
\sum_{\phi_n\in[0,yt_n[}
\sum_{\phi_{n-1}\in[\tilde\phi_n,\tilde\phi_n^+[}\cdots
\sum_{\phi_2\in[\tilde\phi_3,\tilde\phi_3^+[}e_{m_1}(\P_{\phi_2}/\P_{\phi_2}^+)
\le e_{m_1}((R/p_2)_{p_1})P_{R_{p_2}}(y)=P_{R_{p_1}}(y)
\end{equation}
for $y\gg 0$. We apply Theorem \ref{TheoremA2} to the valuations $\nu=\nu_1$ and $\nu_2$ and $p_2= p_1\subset p_0$ to obtain for $\phi_2\in[\tilde\phi_3,\tilde\phi_3^+[$ (or $\phi_2\in [\tilde\phi_n,\tilde\phi_n^++t_ny[$ if $n=3$), 
\begin{equation}\label{eqL11}
\sum_{\phi_1\in[\tilde\phi_2,\tilde\phi_2^+[}
e_{m_0}((\P_{\phi_1}/\P_{\phi_1}^+)_{p_0})\le
e_{m_0}((R/p_1)_{p_0})e_{m_1}((\P_{\phi_2}/\P_{\phi_2}^+)_{p_1}).
\end{equation}
Now sum over (\ref{eqL10}) and (\ref{eqL11}) to obtain the formula for $\nu$.

\end{proof}

\begin{corollary} In the special case where $\nu_1$ is a valuation of rank one and $\nu_2$ is the trivial valuation, we have $p_2=0$, $\Gamma_2=0$, $\tilde \phi_2=0\in \Gamma_1$ and  the inequality of Theorem \ref{TheoremA2} reduces to:
\begin{equation}\label{eq30}
\begin{array}{l}
\sum_{\phi_1\in[0,y_1t_1[}e_{m_0}(\P_{\phi}(R_{p_0})/\P_{\phi}^+(R_{p_0}))\\ \leq
(1+\epsilon)e_{m_0}((R/p_1)_{p_0})\frac{e_{m_1}(R_{p_1})}{(\hbox{\rm dim}R_{p_1})!}y_1^{\hbox{\rm dim}R_{p_1}}
\end{array}
\end{equation}
for $y_1>s(\epsilon)$.
\end{corollary}
\vskip .2truein
\begin{corollary} Taking $p_0=p_1$, we deduce  that when $p_1\neq p_2$ we have for $y_1>s(\epsilon,\phi_2)$ an inequality
$$
\begin{array}{l}
\#(\lambda_1^{-1}(\phi_2)\cap S^R(\nu)\cap [\tilde \phi_2,\tilde\phi_2+y_1t_1[ )\\ \leq
(1+\epsilon)\frac{e_{m_1}(\P_{\phi_2}(R_{p_1})/\P_{\phi_2}^+(R_{p_1}))}{(\hbox{\rm dim}(R/p_2)_{p_1})!}y_1^{\hbox{\rm dim}(R/p_2)_{p_1}}.
\end{array}
$$
When $p_1=p_2$, we have
$$
\#(\lambda_1^{-1}(\phi_2)\cap S^R(\nu))\le \hbox{\rm length}_{R_{p_1}}(\P_{\phi_2}(R_{p_1})/\P_{\phi_2}^+(R_{p_1}))<\infty.
$$
\par\noindent
\vskip .2truein
\end{corollary}
Let $$
(0)=p_{n+1}\subset p_n\subseteq \cdots \subseteq p_1
$$
be the centers of the valuations with which $\nu$ is composed. Define $I=\{i\in \{1,\ldots ,n\}/ p_i\neq p_{i+1}\}$ and note that $n\in I$. By \cite{ZS} (Appendix 3) we know that if $i\notin I$, if we define $$\phi_i^+=\hbox{\rm min}\{\nu_i(f)| f\in R\ \  \hbox{\rm   and } \nu_i(f)>\phi_i\},$$ 
the intersection $S^R(\nu)\cap [\tilde\phi_i, \widetilde {\phi_i^+}[$ is finite. Let us agree that in this case,  for large $y_i$ the interval $[\tilde\phi_{i+1},\tilde\phi_{i+1}+t_iy_i[$ coincides with this intersection. Remember also that if $i\notin I$ we have $\hbox{\rm dim}(R/p_{i+1})_{p_i}=0$ and $e_{m_i}((R/p_{i+1})_{p_i})=(\hbox{\rm dim}(R/p_{i+1})_{p_i})! =1$.
\begin{theorem}\label{Theorem1} Let $R$ be a local domain and $\nu$ a valuation of $R$ which is of rank $n$. There exist functions
$s_n(\epsilon)$ and $s_i(\epsilon,y_{i+1},y_{i+2},\ldots ,y_n)$ for $1\le i\le n-1$, such that, using the notations and conventions introduced above, 
we have 
$$
\begin{array}{l}
\sum_{\phi_n\in [0,t_ny_n[}\,\,\,\,\,\sum_{\phi_{n-1}\in [
\tilde\phi_n,\tilde\phi_n+t_{n-1}y_{n-1}[}\,\,\,\,\,\cdots \,\,\,\,\,\sum_{\phi_1\in [
\tilde\phi_{2},\tilde\phi_{2}+t_1y_1[}e_{m_0}((\P_{\phi_1}/\P^+_{\phi_1})_{p_0})
\\
\\
\le (1+\epsilon)\frac{\prod_{i=0}^n e_{m_i}((R/p_{i+1})_{p_i})}{\prod_{i=1}^n(\hbox{\rm dim}(R/p_{i+1})_{p_i})!} \prod_{i=1}^ny_i^{\hbox{\rm dim}(R/p_{i+1})_{p_i}}
\end{array}
$$
\vskip .2truein
for $y_n,y_{n-1},\ldots,y_1\in\NN$ satisfying
$$
y_n\ge s_n(\epsilon), y_{n-1}\ge s_{n-1}(\epsilon,y_n),\ldots,y_1\ge s_1(\epsilon,y_2,\ldots ,y_n).
$$
\end{theorem}
\begin{proof} The proof of this formula is by induction on the rank $n$ of the valuation $\nu$. We first prove the formula in the case when $n=1$. We apply (\ref{eq30}) to the ring $R_{p_0}$ and observe that for 
$\phi_1\in\Gamma_1$,
$\P_{\phi_1}(R_{p_0})\cong (\P_{\phi_1})_{p_0}$,
to obtain
$$
\sum_{\phi_1\in [0,t_1y_1[}e_{m_0}((\P_{\phi_1}/\P_{\phi_1}^+)_{p_0})
\le
(1+\epsilon)\frac{e_{m_0}((R/p_1)_{p_0})}{(\hbox{\rm dim }R_{p_1})!}e_{m_1}(R_{p_1})y_1^{\hbox{\rm dim }R_{p_1}}
$$
for $y_1\ge s_1(\epsilon)$, which is the formula for $n=1$.

We now assume that the formula is true for valuations of rank $<n$. We will derive the formula for a rank $n$ valuation $\nu$.   We apply the formula to the rank $n-1$ valuation $\nu_2$ which $\nu$ is composite with, and the chain of prime ideals
$$
(0)=q_n\subset q_{n-1}\subset \cdots\subset q_1\subset q_0
$$
where $q_{n-1}=p_{n},\ldots,q_1=p_2$ are the centers on $R$ of the successive valuations
$\nu_n,\ldots,\nu_2$ with which $\nu_2$ is composed, and $q_0=p_1$ is the new "mute" prime ideal. We obtain the inequality
\begin{equation}\label{eq6}
\sum_{\phi_n\in[0,t_ny_n[}\cdots\sum_{\phi_2\in [\tilde\phi_3,\tilde \phi_3+t_2y_2[}
e_{m_1}((\P_{\phi_2}/\P_{\phi_2}^+)_{p_1})\le
(1+\epsilon_1)\frac{\prod_{i=1}^ne_{m_i}((R/p_{i+1})_{p_i})}
{\prod_{i=2}^n(\hbox{\rm dim }(R/p_{i+1})_{p_i})!}
\prod_{i=2}^ny_i^{\mbox{dim }(R/p_{i+1})_{p_i}}
\end{equation}
for 
$$
y_n\ge s_n(\epsilon_1), y_{n-1}\ge s_{n-1}(\epsilon_1,y_n),\ldots,y_2\ge s_2(\epsilon_1,y_3,\ldots, y_{n-1}).
$$
We  apply Theorem \ref{TheoremA2} to  the valuations $\nu=\nu_1$ and $\nu_2$, 
and $p_2\subseteq p_1\subseteq p_0$, to obtain for $\phi_2\in[\tilde\phi_3+t_2y_2[$,\par\noindent 
In the case where $p_2\neq p_1$
\begin{equation}\label{eq7}
\sum_{\phi_1\in[\tilde\phi_2,\tilde\phi_2+y_1t_1[}e_{m_0}((\P_{\phi_1}/\P_{\phi_1}^+)_{p_0})
\le
(1+\epsilon_1)\frac{e_{m_0}((R/p_1)_{p_0})}
{(\mbox{dim }(R/p_2)_{p_1})!}e_{m_1}((\P_{\phi_2}/\P_{\phi_2}^+)_{p_1})
y_1^{\mbox{dim }(R/p_2)_{p_1}}
\end{equation}
for $y_1>s(\epsilon_1,\phi_2)$. Since $\#(S^R(\nu)\cap [\tilde\phi_3,\tilde\phi_3+t_2y_2[\,)<\infty$, we may define
$$
s_1(\epsilon_1,y_2,y_3,\ldots, y_n)=\mbox{max}\{s(\epsilon_1,\phi_2)\mid\phi_2\in[\tilde\phi_3,\tilde\phi_3+t_2y_2[, \phi_3\in [\tilde\phi_4,\tilde\phi_4+t_3y_3[,\ldots, \phi_n\in[0,t_ny_n[
\}.
$$

In the case where $p_2=p_1$ we have by theorem \ref{TheoremA2}, b)  the equality
$$
\sum_{\phi_1\in[\tilde\phi_2,\widetilde{\phi_2^+[}}e_{m_0}((\P_{\phi_1}/\P_{\phi_1}^+)_{p_0})= e_{m_0}((\P_{\phi_2}/\P_{\phi_2}^+)_{p_0})=e_{m_0}((R/p_1)_{p_0})e_{m_1}((\P_{\phi_2}/\P_{\phi_2}^+)_{p_1}),
$$
and define $s_1(\epsilon_1,y_2)=1$.

 \par\noindent
Finally, we  set
$$
\epsilon_1=2^{\frac{1}{2}\mbox{log}_2(1+\epsilon)}-1
,$$
so that $(1+\epsilon_1)^2=1+\epsilon$, and sum over $(\ref{eq7})$ and $(\ref{eq6})$ after multiplication by the appropriate factor to obtain the desired formula for $\nu$.
\end{proof}

As an immediate corollary, we obtain
\begin{corollary}\label{Cor5}The sum
$$\sum_{\phi_n\in [0,t_ny_n[}\,\,\,\,\,\sum_{\phi_{n-1}\in [
\tilde\phi_n,\tilde\phi_n+t_{n-1}y_{n-1}[}\cdots \sum_{\phi_1\in [
\tilde\phi_{2},\tilde\phi_{2}+t_1y_1[}e_{m_0}((\P_{\phi_1}/\P^+_{\phi_1})_{p_0})$$
is bounded for $y_1\gg y_2\gg \cdots\gg y_n\gg 0$  by a function which behaves asymptotically as
$$\frac{\prod_{i=0}^n e_{m_i}((R/p_{i+1})_{p_i})}{\prod_{i=1}^n(\hbox{\rm dim}(R/p_{i+1})_{p_i})!} \prod_{i=1}^ny_i^{\hbox{\rm dim }(R/p_{i+1})_{p_i}}.$$
 \par\noindent
\end{corollary}
Using the notations of Definition \ref{Def1}, and the conventions preceding Theorem \ref{Theorem1}, define the pseudo-boxes 
$$B_\Gamma (y_1,\ldots, y_n)=\{\phi\in \Gamma/\phi\in [
\tilde\phi_{2},\tilde\phi_{2}+t_1y_1[, \phi_2\in [\tilde\phi_3, \tilde\phi_3+t_2y_2[, \ldots, \phi_n\in [0,t_ny_n[\}.$$ Then we have:
\begin{corollary}\label{Cor6} For $y_1\gg y_2\gg \cdots\gg y_n\gg 0$ the number 
$\#(S^R(\nu)\bigcap B_\Gamma(y_1,\ldots, y_n))$ is bounded by the same function as in Corollary \ref{Cor5}.
\end{corollary}
\begin{Remark} 1) The only centers $p_i$ which contribute to the right hand side of the inequalities are those for which the inclusion $p_{i+1}\subseteq p_i$ is strict. \par\noindent
2) The total degree of the monomial appearing on the right hand side is $\hbox{\rm dim }R-\hbox{\rm dim }R/p_1$, which is $\hbox{\rm dim }R$ in the case where $\nu$ is centered at the maximal ideal $m_R$.
\end{Remark}

We now give an application of  Theorem \ref{Theorem1}.
Suppose that $\nu$ is a rank 2 valuation dominating a local domain $R$. Let $\Gamma_2$ be the value group of the composed valuation $\nu_2$ of the quotient field of $R$, and let $p_2$ be the center of $\nu_2$ on $R$, $t_1=\nu(m_R)$. 
Theorem \ref{TheoremA2} gives us a family of growth conditions for $\phi_2\in\Gamma_2$ on $S^R(\nu)\cap [\tilde\phi_2,nt_1[$ for $n$ sufficiently large.  To be precise, Theorem \ref{TheoremA2} tells us that for each $\phi_2\in\Gamma_2$, there exist functions $d(\phi_2)$ and $s(\phi_2)\in\NN$
such that 
\begin{equation}\label{eq07}
\#(S^R(\nu)\cap[\tilde\phi_2,\tilde\phi_2+nt_1[\,)<d(\phi_2)n^{\mbox{dim }R/p_2}
\end{equation}
for $n>s(\phi_2)$.

\begin{Example}\label{Example3}
For every natural number $s\ge 3$, there exists a rank 2, well ordered subsemigroup $T$ of the positive part of  $(\ZZ\times\QQ)_{\mbox{lex}}$, which is of ordinal type $\omega^2$ and satisfies the restrictions (\ref{eq07}) for all $\phi_2\in\NN$, but is not the semigroup of a valuation dominating 
an $s$ dimensional local domain.
\end{Example}

\begin{proof} Let $r=s-2$. Define a subsemigroup of $\QQ_{\ge 0}$ by 
$$
\begin{array}{lll}
S&=&\{(2^m+j)+\frac{\alpha}{2^{(m+1)r}}\mid
m,j,\alpha\in\NN,\\
&&\,\,\,\,\, 0\le j< 2^m, 0\le\alpha<2^{(m+1)r}\}.
\end{array}
$$
Suppose that $n$ is a positive integer. Then there exists a unique expression $n=2^m+j$ with
$0\le j< 2^m$.   We have
$$
\#(S\cap [n,n+1[\,)=2^{(m+1)r},
$$
and since $2^m\le n<2^{m+1}$,
\begin{equation}\label{eqN2}
n^r<\#(S\cap [n,n+1[\,)\le 2^r n^r.
\end{equation}

For $y$ a positive integer, let $f(y)=\sum_{n=1}^{y-1}n^r$. The first difference function $f(y+1)-f(y)=y^r$ is a polynomial of degree $r$ in $y$. Thus
$f(y)$ is a polynomial of degree $r+1$ in $y$, with positive leading coefficient.
From
$$
\#(S\cap [0,y\,)=\sum_{n=1}^{y-1}\#(S\cap [n,n+1[\,)
$$
and (\ref{eqN2}), we deduce that
\begin{equation}\label{eqN1}
f(y)<\#(S\cap [0,y[\,)\le 2^rf(y).
\end{equation}

Suppose that $c\in\NN$. Then 
$\#((\frac{1}{c}S)\cap[0,y[\,)=\#(S\cap[0,cy[\,)$.
 Thus 
\begin{equation}\label{eq20}
f(cy)< \#((\frac{1}{c}S)\cap[0,y[\,) \le 2^rf(cy).
\end{equation}

For $i\in\NN$, let 
\begin{equation}\label{eq22}
c(i)=\left\{\begin{array}{ll}
1&\mbox{ if }i=0\\
i &\mbox{ if }i\ge 1.
\end{array}\right.
\end{equation}
 Let 
$$
T=\bigcup_{m\in\NN}\{m\}\times(\frac{1}{c(m)}S)\subset (\ZZ\times\QQ)_{\mbox{lex}}.
$$
$T$ is a well ordered subsemigroup of $(\ZZ\times\QQ)_{\mbox{lex}}$, of ordinal type $\omega^2$.

Suppose that $T$ is the semigroup $S^R(\nu)$ of a valuation $\nu$ dominating an $s=r+2$ dimensional  local domain $R$.
Then $\nu$ has rank 2. Let $\nu_2$ be the composed valuation
$\nu_2(f)=\pi_1(\nu(f))$ for $f\in R$, where $\pi_1:\ZZ\times\QQ\rightarrow\ZZ$ is the first projection. By assumption, the center of $\nu$ on $R$ is the maximal ideal $m_R$ of $R$. Let $p_2$ be the center of
$\nu_2$ on $R$. We see from an inspection of $T$ that
$t_1=\nu(m_R)=(0,1)$ and $t_2=\nu_2(p_2)=1$. 
Further, $\tilde \phi_2=(\phi_2,\frac{1}{c(\phi_2)})$ for all $\phi_2\in\ZZ_+$. 
Observe that for all $\phi_2\in\ZZ_+$, and $y_1\in\NN$,
$$
\#(T\cap[\tilde\phi_2,\tilde\phi_2+y_1t_1[\,)=
\#(T\cap \{\phi_2\}\times [0, y_1[\,)=\#(\frac{1}{c(\phi_2)}S\cap [0, y_1[\,).
$$
From (\ref{eq20}), we see that 
\begin{equation}\label{eq21}
 f(c(\phi_2)y_1)< \#(T\cap \{\phi_2\}\times [0, y_1[\,)
\le  2^rf(c(\phi_2)y_1).
\end{equation}

Thus $T$ satisfies the  growth conditions (\ref{eq07}) on a local domain $R$  with 
$\mbox{dim } R/p_2\ge r+1$. Since $T$ has rank 2, we must have that 
$\mbox{dim } R\ge \mbox{dim } R/p_2+1$.

Since we are assuming that $R$ has dimension $s=r+2$,  we have
that $\mbox{dim } R/p_2=r+1$ and $\mbox{dim } R_{p_2}=1$. Since $\nu_2$ is a discrete rank 1 valuation dominating $R_{p_2}$, this is consistent.

Theorem \ref{Theorem1} tells us that there exists a function $s(y_2)$ and $d\in \ZZ_+$ such that
$$
\#([0,y_2[\,\times [0,y_1[\,\,\cap\, T)\le d y_1^ay_2^b
$$
for $y_1\ge s(y_2)$,
where $a=\mbox{dim } R/p_2=r+1$ and $b=\mbox{dim } R_{p_2}=1$.
From (\ref{eq21}) and (\ref{eq22}), we see that
$$
f(y_1)+\sum_{i=1}^{y_2-1}f(iy_1)\le \#([0,y_2[\,\times [0,y_1[\,\,\cap\, T).
$$ 
There exists a positive constant $e_1$ such that  $f(y_1)\ge e_1y_1^{r+1}$ for all $y_1\in\NN$, and thus there exists a positive constant $e$ such that
$$
f(y_1)+\sum_{i=1}^{y_2-1}f(iy_1)\ge ey_1^{r+1}y_2^{r+1}
$$
for $y_1,y_2\in\NN$. Thus we have 
$$
dy_1^ay_2^b=dy_1^{r+1}y_2> ey_1^{r+1}y_2^{r+1}
$$
for all large $y_2$, which is impossible.

\end{proof}


\section{Wild behavior of the tilde function}

The tilde function $\tilde\phi$, defined in Definition \ref{Def1}, gives critical information about the behavior of valuations of rank larger than one.
This is illustrated by its role in the statement of Theorem \ref{Theorem1}, which 
 shows that there is some order in the behavior of 
semigroups of higher rank valuations. However, the sums in this theorem are
all defined starting from the functions $\tilde\phi$. This function can be
extremely chaotic, as we will illustrate in this section.

We will give examples of rank two valuations, showing that $\tilde\phi$ can decrease arbitrarily
fast as $\phi$ increases (Example \ref{ExW1}), $\tilde\phi$ can increase 
arbitrarily fast as $\phi$ increases (Example \ref{ExW2}), and that $\tilde\phi$
can jump back and forth from   negative numbers which decrease arbitrarily fast to  positive numbers which increase arbitrarily fast as $\phi$ increases (Example \ref{ExW3}). These properties are all independent of order preserving isomomorphism of the value group.

To construct our examples, we will make use of the following technical lemma,
and some variants of it.  This lemma is a generalization of the notion of
generating sequences of valuations on regular local rings of dimension 2 (\cite{M}, \cite{S}, \cite{FJ}). 

\begin{lemma}\label{LemmaW1} Suppose that $\sigma:\ZZ_+\rightarrow \NN$ is a function.
Let $K(x,y,z)$ be a rational function field 
in three variables over a field $K$.
Set 
$$
P_0=x, P_1=y
$$
and
$$
 P_{i+1}=z^{\sigma(i)}P_i^2-P_0^{2^{i+1}}P_{i-1}
$$
for $i\ge 1$.

Define, by induction on $i$,
$\eta_0=1$, and 
\begin{equation}\label{eqW13}
\eta_{i+1}=2\eta_i+\frac{1}{2^{i+1}}
\end{equation}
for $i\ge 0$. Define $\gamma_0=0$ and 
\begin{equation}\label{eqW14}
\gamma_i=-(\frac{\sigma(1)}{2^i}+\frac{\sigma(2)}{2^{i-1}}+\cdots+\frac{\sigma(i)}{2})
\end{equation}
for $i>0$.

\begin{enumerate}
\item[1.] Suppose that $f(x,y,z)\in K(z)[x,y]$. Then for $l\in\NN$ such that $\mbox{deg}_y f<2^l$, there  is a unique expansion 
\begin{equation}\label{eqW2}
f=\sum_{\alpha}a_{\alpha}(z)x^{\alpha_0}P_1^{\alpha_1}\cdots P_l^{\alpha_l}
\end{equation}
where $a_{\alpha}(z)\in K(z)$ and the sum is over $\alpha=(\alpha_0,\alpha_1,\ldots,\alpha_l)\in\NN\times\{0,1\}^l$.
\item[2.]
For $f\in K(z)[x,y]$, define from the expansion (\ref{eqW2}), 
\begin{equation}\label{eqW4}
\nu(f) =\mbox{min}_{\alpha}\{(0,\mbox{ord}_z(a_{\alpha}(z))+ \sum_{i=0}^{\ell}\alpha_i(\eta_i,\gamma_i)\}\in 
(\frac{1}{2^{\infty}}\ZZ\times \frac{1}{2^{\infty}}\ZZ)_{\mbox{lex}}.
\end{equation}\par\noindent
Then $\nu$ defines a rank 2 valuation on $K(x,y,z)$,  which is composite with a rank
1 valuation $\nu_2$ of $K(x,y,z)$. The value group of $\nu_2$ is 
$\frac{1}{2^{\infty}}\ZZ=\bigcup_{i=1}^{\infty}\frac{1}{2^i}\ZZ$.

The valuation $\nu$ dominates the local ring $R=K[x,y,z]_{(x,y,z)}$ 
and the center of $\nu_2$ on $R$ is the prime ideal $(x,y)$.
\end{enumerate}

\end{lemma}

\begin{proof}
We have that $\mbox{deg}_y P_i=2^{i-1}$ for $i\ge 1$. Suppose that $f\in K(z)[x,y]$
and $\ell$ is such that $\mbox{deg}_y f<2^{\ell}$.
By the euclidean algorithm, we have a unique expansion
$$
f=g_0(x,y)+g_1(x,y)P_l
$$
with $g_0,g_1\in K(z)[x,y]$ and $\mbox{deg}_y g_0<2^{l-1}$, $\mbox{deg}_y g_1<2^{l-1}$.
Iterating, we have a unique expansion of $f$ of the form of (\ref{eqW2}).

We have 
\begin{equation}\label{eqW5}
\nu(P_i)=(\eta_i,\gamma_i)
\end{equation}
and 

\begin{equation}\label{eqW6}
\nu(z^{\sigma(i)}P_i^2)=\nu(P_0^{2^{i+1}}P_{i-1})<\nu(P_{i+1})
\end{equation}
for all $i$.

Observe that for 
$$
\alpha=(\alpha_0,\ldots,\alpha_{\ell}), \beta=(\beta_0,\ldots,\beta_{\ell})\in\NN\times\{0,1\}^{\ell},
$$
\begin{equation}\label{eqW3}
\sum_{i=0}^{\ell}\alpha_i\eta_i=\sum_{i=0}^{\ell}\beta_i\eta_i
\end{equation}
implies $\alpha=\beta$.

The function $\nu$ defined by (\ref{eqW4}) thus has the property that there is 
a unique term in the expansion (\ref{eqW2}) for which the minimum (\ref{eqW4}) is achieved.  We will verify that $\nu$ defines a valuation on $K(x,y,z)$.
Suppose that $f,g\in K(z)[x,y]$. Let 

\begin{equation}\label{eqW7}
f=\sum_{\alpha}a_{\alpha}(z)x^{\alpha_0}P_1^{\alpha_1}\cdots P_{\ell}^{\alpha_{\ell}}
\end{equation}
and 
\begin{equation}\label{eqW8}
g=\sum_{\beta}a_{\beta}(z)x^{\beta_0}P_1^{\beta_1}\cdots P_{\ell}^{\beta_{\ell}}
\end{equation}
be the expressions of $f$ and $g$ of the form (\ref{eqW2}).

$$
f+g=\sum_{\alpha}(a_{\alpha}(z)+b_{\alpha}(z))x^{\alpha_0}P_1^{\alpha_1}\cdots
P_{\ell}^{\ell}
$$
is the expansion of $f+g$ of the form (\ref{eqW2}). Since $\mbox{ord}_z (a_{\alpha}(z)+b_{\alpha}(z))\ge \mbox{min}\{\mbox{ord}_z(a(z)),\mbox{ord}_z(b(z))\}$
for all $\alpha$, we have that 
$$
\nu(f+g)\ge\mbox{min}\{\nu(f),\nu(g)\}.
$$

We will now show that $\nu(fg)=\nu(f)+\nu(g)$.

Suppose that 
\begin{equation}\label{eqW9}
s=\sum_{\delta}d_{\delta}(z)x^{\delta_0}P_1^{\delta_1}\cdots P_{\ell}^{\delta_{\ell}}
\end{equation}
is an expansion, with $d_{\delta}\in K(z)$, and $\delta=(\delta_0,\ldots,\delta_{\ell})\in\NN^{\ell+1}$ for all $\delta$. 
We define
$$
\Lambda(s)=\mbox{min}_\delta\{(0,\mbox{ord}_zd_{\delta}(z))+\sum_{i=0}^{\ell}\delta_i(\eta_i,\gamma_i)\}.
$$
Observe that if $s$ is an expansion of the form (\ref{eqW2}); that is,
$\delta\in\NN\times\{0,1\}^{\ell}$ for all $\delta$, then
$\Lambda(s)=\nu(s)$.

Let 
$$
c_{\epsilon}(z)=\sum_{\alpha+\beta=\epsilon}a_{\alpha}(z)b_{\beta}(z).
$$
Let 
\begin{equation}\label{eqW10}
s_0=\sum_{\epsilon}c_{\epsilon}(z)x^{\epsilon_0}P_1^{\epsilon_1}\cdots P_{\ell}^{\epsilon_{\ell}}.
\end{equation}
$s_0$ is an expansion of the form (\ref{eqW9}), and $fg=s_0$. To simplify the indexing later on in the proof, we observe that we can initially take $\ell$ as large as we like.

Let $\alpha'$ (in the expansion (\ref{eqW7})) be such that 
$$
\nu(f)=\nu(a_{\alpha'}(z)x^{\alpha_0'}P_1^{\alpha_1'}\cdots P_{\ell}^{\alpha_{\ell}'})
$$
and let $\beta'$ (in the expansion (\ref{eqW8})) be such that
$$
\nu(g)=\nu(b_{\beta'}(z)x^{\beta_0'}P_1^{\beta_1'}\cdots P_{\ell}^{\beta_{\ell}'}).
$$
Let $\epsilon'=\alpha'+\beta'$. Then 
$$
c_{\epsilon'}(z)x^{\epsilon_0'}P_1^{\epsilon_1'}\cdots P_{\ell}^{\epsilon_{\ell}'}
$$
is the only term in $s_0$ which achieves  the minimum  $\Lambda(s_0)$.
We have that $c_{\epsilon'}(z)=a_{\alpha'}(z)b_{\beta'}(z)$ and 
\begin{equation}\label{eqW11}
\Lambda(s_0)=\nu(f)+\nu(g).
\end{equation}
If $\epsilon\in\NN\times\{0,1\}^{\ell}$ whenever $c_{\epsilon}\ne 0$, then we can can compute $\nu(fg)=\Lambda(s_0)$ and we are done. Otherwise, there exists an $i\ge 1$
such that there exists an $\epsilon$ with $\epsilon_i\ge 2$ and $c_{\epsilon}(z)\ne 0$. We then substitute the identity:
\begin{equation}\label{eqW12}
P_i^2=\frac{1}{z^{\sigma(i)}}P_{i+1}+\frac{1}{z^{\sigma(i)}}x^{2^{i+1}}P_{i-1}
\end{equation}
into $s_0$ to obtain an expansion of the form (\ref{eqW9}), where all terms
$$
c_{\epsilon}(z)x^{\epsilon_0}P_1^{\epsilon_1}\cdots P_{\ell}^{\gamma_{\ell}}
$$
with $\epsilon_i\ge 2$
are modified to the sum of two terms
$$
\frac{c_{\epsilon}(z)}{z^{\sigma(i)}}x^{\epsilon_0+2^{i+1}}P_1^{\epsilon_1}\cdots
P_{i-1}^{\epsilon_{i-1}+1}P_i^{\epsilon_i-2}P_{i+1}^{\epsilon_{i+1}}\cdots\P_{\ell}^{\epsilon_{\ell}}
+
\frac{c_{\epsilon}(z)}{z^{\sigma(i)}}x^{\epsilon_0}P_1^{\epsilon_1}\cdots
P_{i-1}^{\epsilon_{i-1}}P_i^{\epsilon_i-2}P_{i+1}^{\epsilon_{i+1}+1}\cdots\P_{\ell}^{\epsilon_{\ell}}.
$$
Collecting terms with like monomials in $x, P_1,\ldots,P_{\ell}$, we obtain a new expansion
$$
s_1=\sum d_{\delta}(z)x^{\delta_0}P_1^{\delta_1}\cdots P_{\ell}^{\delta_{\ell}}
$$
of $fg$. From the identities (\ref{eqW6}), we see that the minimum $\Lambda(s_1)$
is only obtained by the term $d_{\delta'}(z)x^{\delta_0'}P_1^{\delta_1'}\cdots P_{\ell}^{\delta_{\ell}'}$, where
$$
d_{\delta'}(z)x^{\delta_0'}P_1^{\delta_1'}\cdots P_{\ell}^{\delta_{\ell}'}
=\left\{\begin{array}{ll}
c_{\epsilon'}(z)x^{\epsilon_0'}P_1^{\epsilon_1'}\cdots P_{\ell}^{\epsilon_{\ell}'}
&\mbox{ if $\epsilon'_i<2$ }\\
\frac{c_{\epsilon'}(z)}{z^{\sigma(i)}}x^{\epsilon'_0+2^{i+1}}P_1^{\epsilon'_1}\cdots
P_{i-1}^{\epsilon'_{i-1}+1}P_i^{\epsilon'_i-2}P_{i+1}^{\epsilon'_{i+1}}\cdots\P_{\ell}^{\epsilon'_{\ell}}
&\mbox{ if $\epsilon'_i\ge 2$. }
\end{array}\right.
$$
We have $\Lambda(s_1)=\Lambda(s_0)$. 

By descending induction on the invariants 
$$
n=\mbox{max}\{\delta_1+\cdots+\delta_{\ell}\mid \mbox{some $\delta_i\ge 2$ and $d_{\delta}(z)\ne 0$}\}
$$
and
$$
m=\#\{(\delta_1,\ldots,\delta_{\ell})\in \NN^{\ell}\mid
\delta_1+\cdots+\delta_{\ell}=n,\mbox {some $\delta_i\ge 2$ and $d_{\delta}(z)\ne 0$}\},
$$
making substitutions of the form (\ref{eqW12}), we eventually obtain an expression $s$ of $fg$ of the form (\ref{eqW9}), with 
$(\delta_0,\ldots,\delta_{\ell})\in\NN\times\{0,1\}^{\ell}$ for all $\ell$. 
We then compute $\nu(fg)=\Lambda(s)=\Lambda(s_0)=\nu(f)+\nu(g)$.
We have thus completed the verification that $\nu$ is a valuation.

\end{proof}

\begin{example}\label{ExW1}   Suppose that $f:\NN\rightarrow\ZZ$ is a decreasing function,
and $K$ is a field.
Then there exists a valuation $\nu$  of the three dimensional rational function field $K(x,y,z)$
 with value group 
$(\frac{1}{2^{\infty}}\ZZ\times\ZZ)_{\mbox{lex}}$, which dominates the regular local ring $R=K[x,y,z]_{(x,y,z)}$, such that for any valuation $\omega$ equivalent to $\nu$ with value group $(\frac{1}{2^{\infty}}\ZZ\times\ZZ)_{\mbox{lex}}$, for all sufficiently large
$n\in\NN$, there exists $\lambda\in  \frac{1}{2^{\infty}}\ZZ\cap[0,n[$ such that
$\pi_2(\tilde\lambda)< f(n)$, where $\pi_2:\frac{1}{2^{\infty}}\ZZ\times\ZZ\rightarrow\ZZ$ is the second projection.
\end{example}

\begin{proof}  We choose positive integers $\sigma(i)$ so that 
$$
\gamma_i=-(\frac{\sigma(1)}{2^i}+\frac{\sigma(2)}{2^{i-1}}+\cdots+
\frac{\sigma(i)}{2})< f(i2^{i+3})
$$
for all positive integers $i$, 
where $\eta_i$ are defined by (\ref{eqW13}), and so that $\gamma_i\in\ZZ$.  Let $\nu$ be the valuation dominating $R$ defined by Lemma \ref{LemmaW1}, with this
choice of $\sigma$. The value group of $\nu$ is $(\frac{1}{2^{\infty}}\ZZ\times\ZZ)_{\mbox{lex}}$.

Let $\omega$ be a valuation equivalent to $\nu$ with value group $(\frac{1}{2^{\infty}}\ZZ\times\ZZ)_{\mbox{lex}}$. We have 
$\omega(x)=(a,b)$ for some $a\in \frac{1}{2^{\infty}}\ZZ_+$ and $b\in\ZZ$, and $\omega(z)=(0,c)$ for
some $c\in\ZZ_+$ since convex subgroups have to be preserved under an automorphism of ordered groups. From the relations (\ref{eqW6}) we see that
$\omega(P_i)=\eta_i\omega(x)+\gamma_i\omega(z)$ for $i\ge 1$, which implies $b=0$, since $\pi_2(\omega(P_i))\in\ZZ$. We also have
$\omega(a(z))=\mbox{ord}_z(a)\omega(z)$ for $a(z)\in K(z)$. Let $\omega_2$ be the valuation on $K(x,y,z)$ defined by $\omega_2(f)=\pi_1(\omega(f))$, where 
$\pi_1:\frac{1}{2^{\infty}}\ZZ\times\ZZ\rightarrow\frac{1}{2^{\infty}}\ZZ$ is the first projection.

Let 
$$
e= \lceil a \rceil,\mbox{
and }n_0=\lceil a\rceil 2^{e+2}.
$$

Suppose that $n\ge n_0$. We will find $\lambda\in \frac{1}{2^{\infty}}\ZZ\cap [0,n[$ such that $\pi_2(\tilde\lambda)<f(n)$.

There exists $i\ge e$ such that
$$
\lceil a\rceil 2^{i+2}\le n< \lceil a \rceil 2^{i+3}.
$$
Let $\lambda=a\eta_i$. From $\eta_i=\frac{1}{3}(2^{i+2}-\frac{1}{2^i})$, we obtain 
$$
\lambda=a\eta_i<\lceil a\rceil 2^{i+2}\le n.
$$
Since $c\in\ZZ_+$, $i\ge e$ and $f$ is decreasing, we have 
$$
c\gamma_i\le\gamma_i< f(i2^{i+3})
\le f(\lceil a\rceil 2^{i+3})
< f(n).
$$
Thus 
$$
\pi_2(\tilde \lambda)\le \pi_2(\omega(P_i))=c\gamma_i<f(n).
$$
\end{proof}

\begin{example}\label{ExW2}  Suppose that $g:\NN\rightarrow\ZZ$ is an increasing function,
and $K$ is a field.
Then there exists a valuation $\nu$  of the three dimensional rational function field $K(x,y,z)$
 with value group 
$(\frac{1}{2^{\infty}}\ZZ\times\ZZ)_{\mbox{lex}}$, which dominates the regular local ring $R=K[x,y,z]_{(x,y,z)}$, such that for any valuation $\omega$ equivalent to $\nu$ with value group $(\frac{1}{2^{\infty}}\ZZ\times\ZZ)_{\mbox{lex}}$, for all sufficiently large
$n\in\NN$, there exists $\lambda\in  \frac{1}{2^{\infty}}\ZZ\cap [0,n[$ such that
$\pi_2(\tilde\lambda)> g(n)$, where $\pi_2:\frac{1}{2^{\infty}}\ZZ\times\ZZ\rightarrow\ZZ$ is the second projection.
\end{example}

\begin{proof}
The proof is a variation of the proof of Example \ref{ExW1}. We outline
it here.

We first must establish a modification of Lemma \ref{LemmaW1}. 

Suppose that $\tau:\ZZ_+\rightarrow \NN$ is a function.
Set 
$$
Q_0=x, Q_1=y
$$
and
$$
 Q_{i+1}=Q_i^2-z^{\tau(i)}Q_0^{2^{i+1}}Q_{i-1}
$$
for $i\ge 1$.

As in Lemma \ref{LemmaW1}, 
define by induction on $i$,
$\eta_0=1$, and 
\begin{equation}\label{eqW20}
\eta_{i+1}=2\eta_i+\frac{1}{2^{i+1}}
\end{equation}
for $i\ge 0$. Define $\delta_0=0$ and 
\begin{equation}\label{eqW21}
\delta_i=\frac{\tau(1)}{2^i}+\frac{\tau(2)}{2^{i-1}}+\cdots+\frac{\tau(i)}{2}
\end{equation}
for $i>0$. 

Suppose that $f(x,y,z)\in K[x,y,z]$. Then for $l\in\NN$ such that $\mbox{deg}_y f<2^l$, there  is a unique expansion 
\begin{equation}\label{eqW22}
f=\sum_{\alpha}a_{\alpha}(z)x^{\alpha_0}P_1^{\alpha_1}\cdots P_l^{\alpha_l}
\end{equation}
where $a_{\alpha}(z)\in K[z]$ and the sum is over $\alpha=(\alpha_0,\alpha_1,\ldots,\alpha_l)\in\NN\times\{0,1\}^l$.
This is established as (\ref{eqW2}) in the statement of Lemma \ref{LemmaW1}. We have
a stronger statement which is valid in the polynomial ring $K[x,y,z]$, since
the leading coefficients of the $Q_i$, as polynomials in $y$, have 1 as their leading coefficient.

For $f\in K[x,y,z]$, define from the expansion (\ref{eqW22}), 
\begin{equation}\label{eqW23}
\nu(f) =\mbox{min}_{\alpha}\{ (0,\mbox{ord}_z(a_{\alpha}(z))+\sum_{i=0}^{\ell}\alpha_i(\eta_i,\delta_i)\}\in 
(\frac{1}{2^{\infty}}\ZZ\times \frac{1}{2^{\infty}}\ZZ)_{\mbox{lex}}.
\end{equation}
Then $\nu$ defines a rank 2 valuation on $K(x,y,z)$,  which is composite with a rank
1 valuation $\nu_2$ of $K(x,y,z)$. The value group of $\nu_2$ is 
$\frac{1}{2^{\infty}}\ZZ=\bigcup_{i=1}^{\infty}\frac{1}{2^i}\ZZ$. 

$\nu$ dominates the local ring $R=K[x,y,z]_{(x,y,z)}$ 
and the center of $\nu_2$ on $R$ is the prime ideal $(x,y)$.

We have $\nu(Q_i)=(\eta_i,\delta_i)$ and 
\begin{equation}\label{eqW24}
\nu(Q_i^2)=\nu(z^{\tau(i)}Q_0^{2^{i+1}}Q_{i-1})<\nu(Q_{i+1})
\end{equation}
for all $i$.

We now construct the example. We choose positive integers $\tau(i)$ so that
$$
\delta_i=\frac{\tau(1)}{2^i}+\frac{\tau(2)}{2^{i-1}}+\cdots+\frac{\tau(i)}{2}
>g(i2^{i+3})
$$
for all positive integers $i$, and $\delta_i\in\ZZ$ for all $i$. Let $\nu$ be the valuation constructed above, which dominates $R$. The value group of $\nu$ is
$(\frac{1}{2^{\infty}}\ZZ\times\ZZ)_{\mbox{lex}}$.

Let $\omega$ be a valuation equivalent to $\nu$ with value group $(\frac{1}{2^{\infty}}\ZZ\times\ZZ)_{\mbox{lex}}$. We have 
$\omega(x)=(a,b)$ for some $a\in\frac{1}{2^{\infty}}\ZZ_+$ and $b\in\ZZ$, and $\omega(z)=(0,c)$ for
some $c\in\ZZ_+$.  From the relations (\ref{eqW24}) we see that
$\omega(Q_i)=\eta_i\omega(x)+\delta_i\omega(z)$ for $i\ge 1$, which implies $b=0$, since $\pi_2(\omega(Q_i))\in\ZZ$. We also have 
$\omega(a(z))=\mbox{ord}_z(a)\omega(z)$ for $a(z)\in K(z)$. Let $\omega_2$ be the valuation on $K(x,y,z)$ defined by $\omega_2(f)=\pi_1(\omega(f))$, where
$\pi_1:\frac{1}{2^{\infty}}\ZZ\times\ZZ\rightarrow\frac{1}{2^{\infty}}\ZZ$ is the first projection.

Let 
$$
e=\lceil a \rceil\mbox{
and }n_0=\lceil a\rceil 2^{e+2}.
$$

Suppose that $n\ge n_0$. We will find $\lambda\in \frac{1}{2^{\infty}}\ZZ\cap [0,n[$ such that $\pi_2(\tilde\lambda)>g(n)$.

There exists $i\ge e$ such that
$$
\lceil a\rceil 2^{i+2}\le n< \lceil a\rceil 2^{i+3}.
$$
Let $\lambda=a\eta_i$. From $\eta_i=\frac{1}{3}(2^{i+2}-\frac{1}{2^i})$, we obtain 
$$
\lambda=a\eta_i<\lceil a\rceil 2^{i+2}\le n.
$$
Since $c\in\ZZ_+$, $i\ge e$ and $g$ is increasing, we have 
$$
c\delta_i\ge\delta_i> g(i2^{i+3})
\ge g(\lceil a\rceil 2^{i+3})
> g(n).
$$
Thus 
$$
\pi_2(\tilde \lambda)= \pi_2(\omega(Q_i))=c\delta_i>g(n).
$$

\end{proof}

\begin{example}\label{ExW3}  Suppose that 
$f:\NN\rightarrow\ZZ$ is a decreasing function,
$g:\NN\rightarrow\ZZ$ is an increasing function,
and $K$ is a field.
Then there exists a rank 2 valuation $\nu$  of the five dimensional rational function field $K(x,y,u,v,z)$
 with value group 
$(H\times\ZZ)_{\mbox{lex}}$, where $H=(\frac{1}{2^{\infty}}\ZZ+\frac{1}{2^{\infty}}\ZZ\sqrt{2})\subset\RR$, which dominates the regular local ring $R=K[x,y,u,v,z]_{(x,y,u,v,z)}$, such that for any valuation $\omega$ equivalent to $\nu$ with value group $(H\times\ZZ)_{\mbox{lex}}$, for all sufficiently large
$n\in\NN$, there exists $\lambda_1\in  H\cap [0,n[$ such that
$\pi_2(\tilde\lambda_1)< f(n)$and there exists $\lambda_2\in  H\cap [0,n[$ such that
$\pi_2(\tilde\lambda_2)> g(n)$, where $\pi_2:\frac{1}{2^{\infty}}\ZZ\times\ZZ\rightarrow \frac{1}{2^{\infty}}\ZZ$ is the second projection.
\end{example}

\begin{proof} We need an extension of the method of Lemma \ref{LemmaW1} for
constructing valuations which we first outline. Suppose that
$\sigma:\ZZ_+\rightarrow \NN$ and $\tau:\ZZ_+\rightarrow \NN$ are functions.
Define
$P_0=x$, $P_1=y$ and $P_{i+1}=z^{\sigma(i)}P_i^2-P_0^{2^{i+1}}P_{i-1}$ for $i\ge 1$.  Define $Q_0=u$, $Q_1=v$ and $Q_{i+1}=Q_i^2-z^{\tau(i)}Q_0^{2^{i+1}}Q_{i-1}$
for $i\ge 1$.

Suppose that $f\in K[z,x,y,u,v]$ and $\mbox{deg }f<2^l$. Then there is a unique expansion
$$
f=\sum_{\beta}g_{\beta}(z,x,y)u^{\beta_0}Q_1^{\beta_1}\cdots Q_{\ell}^{\beta_{\ell}}
$$
where the sum is over $\beta=(\beta_0,\beta_1,\ldots,\beta_{\ell})\in\NN\times\{0,1\}^{\ell}$
and $g_{\beta}\in K[z,x,y]$ for all $\beta$ (since the leading coefficient of each $Q_i$ with respect to $y$ is 1).  Each $g_{\beta}(z,x,y)$ has a unique
expansion
$$
g_{\beta}=\sum_{\alpha}a_{\alpha,\beta}(z)x^{\alpha_0}P_1^{\alpha_1}\cdots P_{\ell}^{\alpha_{\ell}}
$$
where the sum is over $\alpha=(\alpha_0,\alpha_1,\ldots,\alpha_{\ell})\in\NN\times\{0,1\}^{\ell}$
and $a_{\alpha,\beta}(z)\in K(z)$ for all $\alpha,\beta$.

Thus $f$ has a unique expansion 
\begin{equation}\label{eqW25}
f=\sum_{\alpha,\beta}a_{\alpha,\beta}(z)x^{\alpha_0} P_1^{\alpha_1}\cdots P_{\ell}^{\alpha_{\ell}}u^{\beta_0}Q_1^{\beta_1}\cdots Q_{\ell}^{\beta_{\ell}} 
\end{equation}
with $\alpha,\beta\in\NN\times\{0,1\}^{\ell}$, and $a_{\alpha,\beta}(z)\in K(z)$.

Set 
$\eta_0=1$ and 
$$
\eta_{i+1}=2\eta_i+\frac{1}{2^{i+1}}
$$
for $i\ge 0$.  Set $\gamma_0=0$ and
$$
\gamma_i=-(\frac{\sigma(1)}{2^i}+\frac{\sigma(2)}{2^{i-1}}+\cdots+\frac{\sigma(i)}{2})
$$
for $i>0$. Set $\delta_0=0$ and
$$
\delta_i=\frac{\tau(1)}{2^i}+\frac{\tau(2)}{2^{i-1}}+\cdots+\frac{\tau(i)}{2}
$$
for $i>0$.

Let $H=(\frac{1}{2^{\infty}}\ZZ+\frac{1}{2^{\infty}}\ZZ\sqrt{2})\subset \RR$.

We define from the expansion (\ref{eqW25}), 
\begin{equation}\label{eqW26}
\nu(f)=\mbox{min}_{\alpha,\beta}\{(0,\mbox{ord}_z(a_{\alpha,\beta}(z))+\sum\alpha_i(\eta_i,\gamma_i)+\sum\beta_i(\eta_i\sqrt{2},\delta_i)\}\in (H\times\frac{1}{2^{\infty}}\ZZ)_{\mbox{lex}}.
\end{equation}

We have 
$$
\nu(z^{\sigma(i)}P_i^2)=\nu(P_0^{2^{i+1}}P_{i-1})<\nu(P_{i+1})
$$
for all $i$,  $\nu(P_i)=(\eta_i,\gamma_i)$ for all $i$,
$$
\nu(Q_i^2)=\nu(z^{\tau(i)}Q_0^{2^{i+1}}Q_{i-1})<\nu(Q_{i+1})
$$
for all $i$,  $\nu(Q_i)=(\eta_i\sqrt{2},\delta_i)$ for all $i$.
Further, there is a unique term in the expansion (\ref{eqW25}) which achieves the minimum (\ref{eqW26}). We have 
$$
\nu(f)\in \frac{1}{2^{\infty}}\ZZ\times\frac{1}{2^{\infty}}\ZZ
$$
if and only if $f$ has the form
$$
f=a_{\alpha,0}(z)x^{\alpha_0}P_1^{\alpha_1}\cdots P_{\ell}^{\alpha_{\ell}}
+\mbox{ higher value terms},
$$
and 
$$
\nu(f)\in \frac{1}{2^{\infty}}\ZZ\sqrt{2}\times\frac{1}{2^{\infty}}\ZZ
$$
if and only if $f$ has the form
$$
f=a_{0,\beta}(z)u^{\beta_0}Q_1^{\beta_1}\cdots Q_{\ell}^{\beta_{\ell}}
+\mbox{ higher value terms}.
$$

Observe that for all $\beta$, we have  $a_{0,\beta}(z)\in K[z]$ in the expansion
(\ref{eqW25}).

We now construct the valuation of the example. For $i\in\ZZ_{+}$, choose $\sigma(i)\in \NN$ such that
$$
\gamma_i< f(i2^{i+3})
$$
and choose $\tau(i)\in \ZZ_{+}$ so that
$$
g(i2^{i+3})<\delta_i
$$
for all $i\in\ZZ_{+}$, and $\gamma_i,\delta_i\in\ZZ$ for all $i$.
Let $\nu$ be the valuation constructed above, which dominates $R$. The value group of $\nu$ is $(H\times\ZZ)_{\mbox{lex}}$.

Let $\omega$ be a valuation equivalent to $\nu$ with value group $(H\times\ZZ)_{\mbox{lex}}$. There exist $\alpha,\beta,\gamma,\delta\in\frac{1}{2^{\infty}}\ZZ$,
$b_1,b_2\in\ZZ$,  $c\in\ZZ_{+}$, such that
$$
\omega(x)=(a_1,b_1), \omega(u)=(a_2,b_2), \omega(z)=(0,c)
$$
where 
$$
a_1=\alpha+\beta\sqrt{2}>0, a_2=\gamma+\delta\sqrt{2}>0
$$
and $\alpha\delta-\beta\gamma\ne 0$. We have $\omega(P_i)=\eta_i\omega(x)+\gamma_i\omega(z)$  and $\omega(Q_i)=\eta_i\omega(u)+\delta_i\omega(z)$ for $i\ge 1$,
which implies $b_1=b_2=0$, since $\pi_2(\omega(P_i)), \pi_2(\omega(Q_i))\in\ZZ$.

Set 
$$
e=\mbox{max}\{\lceil a_1\rceil,\lceil a_2 \rceil\}.
$$
 Let $n_0=e2^{e+2}$.

Suppose that $n\ge n_0$. We will show that there exists $\lambda_1\in H\cap [0,n[$ such that $\pi_2(\tilde\lambda_1)<f(n)$ and there exists $\lambda_2\in H\cap [0,n[$ such that $\pi_2(\tilde\lambda_2)>g(n)$.

There exists $i\ge e$ such that $e 2^{i+2}\le n<e 2^{i+3}$. Let
$\lambda_1=a_1\eta_i$. From $\eta_i=\frac{1}{3}(2^{i+2}-\frac{1}{2^i})$ we obtain
$$
\lambda_1=a_1\eta_i<\lceil a_1\rceil 2^{i+2}\le e 2^{i+2}\le n.
$$
Since $c\in\ZZ_+$, $i\ge e$ and $f$ is decreasing, we have
$$
c\gamma_i\le\gamma_i< f(i2^{i+3})
\le f(e2^{i+3})
< f(n).
$$
Thus 
$$
\pi_2(\tilde\lambda_1)\le\pi_2(\omega(P_i))=c\gamma_i<f(n).
$$
Let
$\lambda_2=a_2\eta_i$. We have
$$
\lambda_2=a_2\eta_i<\lceil a_2\rceil 2^{i+2}\le e 2^{i+2}\le n.
$$
Since $c\in\ZZ_+$, $i\ge e$ and $g$ is increasing, we have
$$
c\delta_i\ge\delta_i> g(i2^{i+3})
\ge g(e2^{i+3})
>g(n).
$$

Thus 
$$
\pi_2(\tilde\lambda_2)=\pi_2(\omega(Q_i))=c\delta_i>g(n).
$$

\end{proof}

\par\vskip.5truecm\noindent
Dale Cutkosky,  \hfill Bernard Teissier,\par\noindent
202 Mathematical Sciences Bldg,\hfill  Institut Math\'{e}matique de Jussieu\par\noindent

University of Missouri,\hfill UMR 7586 du CNRS,\par\noindent
Columbia, MO 65211 USA\hfill 175 Rue du Chevaleret, 75013 Paris,France \par\noindent
cutkoskys@missouri.edu\hfill teissier@math.jussieu.fr\par\noindent

\end{document}